\documentclass[12pt]{amsart}
\usepackage{amsfonts,latexsym,amsthm,amssymb,graphicx}
\usepackage{wrapfig}
\usepackage[all]{xy}
\usepackage{hyperref}

\usepackage[usenames]{color} 

\DeclareFontFamily{OT1}{rsfs}{}
\DeclareFontShape{OT1}{rsfs}{n}{it}{<-> rsfs10}{}
\DeclareMathAlphabet{\mathscr}{OT1}{rsfs}{n}{it}

\CompileMatrices

\newtheorem{theorem}{Theorem}[section]
\newtheorem{lemma}[theorem]{Lemma}
\newtheorem{corol}[theorem]{Corollary}
\newtheorem{prop}[theorem]{Proposition}
\newtheorem{claim}[theorem]{Claim}

\newtheorem*{thmI}{Theorem I}
\newtheorem*{thmII}{Theorem II}
\newtheorem*{thmIII}{Theorem III}
\newtheorem*{thmIV}{Theorem IV}
\newtheorem*{scholium}{Scholium}

{\theoremstyle{definition} \newtheorem{defin}[theorem]{Definition}}
{\theoremstyle{remark} \newtheorem{remark}[theorem]{Remark}
\newtheorem{example}[theorem]{Example}}

\newcommand{\Abb}{{\mathbb{A}}}
\newcommand{\Cbb}{{\mathbb{C}}}
\newcommand{\Pbb}{{\mathbb{P}}}
\newcommand{\Zbb}{{\mathbb{Z}}}
\newcommand{\TV}{{\Til V}}
\newcommand{\cE}{{\mathscr E}}
\newcommand{\cI}{{\mathscr I}}
\newcommand{\cL}{{\mathscr L}}
\newcommand{\cO}{{\mathscr O}}
\newcommand{\Til}[1]{{\widetilde{#1}}}
\newcommand{\one}{1\hskip-3.5pt1}
\newcommand{\csm}{{c_{\operatorname{SM}}}}
\newcommand{\cFu}{{c_{\operatorname{F}}}}

\newcommand{\qede}{\hfill$\lrcorner$}
\newcommand{\shat}{{\hat s}}

\title{
Chern classes of splayed intersections
}
\author{Paolo Aluffi}
\author{Eleonore Faber}
\address{
Mathematics Department, 
Florida State University,
Tallahassee FL 32306, U.S.A.
}
\email{aluffi@math.fsu.edu}
\address{
Department of Computer and Mathematical Sciences,
University of Toronto at Scarborough,
Toronto, ON M1A 1C4,
Canada
}
\email{efaber@math.toronto.edu}

\begin{document}

\begin{abstract}
We generalize the Chern class relation for the transversal intersection of two nonsingular
varieties to a relation for possibly singular varieties, under a {\em splayedness\/} assumption. 
The relation is shown to hold for both the Chern--Schwartz--MacPherson class and the
Chern--Fulton class. The main tool is a formula for Segre classes of splayed subschemes.
We also discuss the Chern class relation under the assumption that one of the varieties
is a general very ample divisor.
\end{abstract}

\maketitle


\section{Introduction}
\subsection{}
Let $X$, $Y$ be nonsingular subvarieties of a nonsingular complex variety $V$. If $X$ and 
$Y$ intersect properly and transversally, then the intersection $X\cap Y$ is nonsingular,
and an elementary Chern class computation proves that
\begin{equation}\label{eq:transversal}
c(X) \cdot c(Y) = c(TV) \cap c(X\cap Y)\quad,
\end{equation}
where $c(X)$, etc.~denote the push-forward to $V$ of the total (homology) Chern class of the tangent
bundle of $X$, etc.,
and $\cdot$ is the intersection product in $V$. It is natural to ask whether \eqref{eq:transversal}
holds if 
$X$, $Y$, $X\cap Y$ are allowed to be singular. 
In~\cite{MR3106736},
\S3, we proposed the following generalization of~\eqref{eq:transversal}:

\begin{scholium}\label{main}
Let $X$, $Y$ be (possibly singular) subvarieties of a nonsingular variety~$V$. Assume
that $X$ and $Y$ are {\em splayed.\/} Then
\begin{equation}\label{eq:splayed}
\csm(X) \cdot \csm(Y) = c(TV) \cap \csm(X\cap Y)\quad.
\end{equation}
\end{scholium}

In~\eqref{eq:splayed}, $\csm(-)$~denotes the {\em Chern-Schwartz-MacPherson\/} 
class;
this is a natural generalization of the total Chern class to singular varieties, and we silently
push this class forward to the ambient variety $V$.
The $\csm$ class is defined for more general schemes---$X$ and $Y$ could be reducible, and 
should not be required to be pure dimensional. 
The purpose of this note is to investigate~\eqref{eq:splayed} at this level of generality.
For example, we will prove that the Scholium holds for arbitrary splayed hypersurfaces,
and more generally for subschemes satisfying a hypothesis of `strong' splayedness.
We also prove~\eqref{eq:splayed} for splayed subschemes for a different notion of
Chern class defined for arbitrary subschemes of a nonsingular variety.
Finally, we will discuss a `Bertini' statement, according to which~\eqref{eq:splayed}
holds if $X$ is a sufficiently general very ample divisor.

The notion of {\em splayedness\/} was
introduced and studied in the hypersurface case by the second author in~\cite{Faber12}, 
and it is explored further in~\cite{MR3106736}: $X$ and $Y$ are splayed if at each point 
$p$ of the intersection there exist analytic coordinates $(x_1,\dots,x_r,y_1,\dots, y_s)$ 
such that $X$ may be defined by an ideal generated by functions in the coordinates $x_i$ 
and $Y$ by an ideal generated by functions in the coordinates $y_j$. 
We also say that two {\em sets\/} $\{X_1, X_2,\dots\}$ and $\{Y_1,Y_2,\dots\}$ are splayed 
if there are local analytic splittings so that all the $X_i$ are defined in the first set of 
coordinates, and all the $Y_j$ are defined in the second set of coordinates. 
These notions generalize to possibly singular varieties and subschemes the notion of 
proper, transversal intersection of nonsingular varieties. \\

The reader who is not too familiar with characteristic classes may view~\eqref{eq:splayed}
as a very general form of identities involving the topological Euler characteristics of 
$X$, $Y$, $X\cap Y$. For example, let $X$ and $Y$ be splayed surfaces in $\Pbb^3$,
of degrees $d$, $e$ resp.; assume that the Euler characteristic of a general hyperplane
section of $X$, resp.~$Y$ is $a$, resp.~$b$. Then it may be checked that the 
Euler characteristic of the curve $X\cap Y$ is $ea+db-2de$. 
Similarly explicit formulas relate the Euler characteristics of general linear sections 
of $X$, $Y$, $X\cap Y$ if these are subsets of projective space and $X$, $Y$ are splayed
(cf.~\cite{MR3031565}). 
The Scholium reveals the underlying structure of all such identities, and generalizes them 
to splayed subsets in arbitrary nonsingular algebraic varieties.\\

Note that some transversality hypothesis is certainly needed
for \eqref{eq:splayed} to hold, as the following example shows.

\begin{example}
Let $X$ be a nonsingular quadric in $V=\Pbb^3$, and let $Y$ be a plane tangent to $X$.
Then $\csm(X)=2[\Pbb^2]+4[\Pbb^1]+4[\Pbb^0]$ and $\csm(Y)=[\Pbb^2]+3[\Pbb^1]
+3[\Pbb^0]$ (since $X$ and $Y$ are nonsingular, these are simply the push-forward
to $\Pbb^3$ of the total Chern classes of their tangent bundles). Thus, the left-hand
side of~\eqref{eq:splayed} is
\[
(2[\Pbb^2]+4[\Pbb^1]+4[\Pbb^0])\cdot ([\Pbb^2]+3[\Pbb^1]+3[\Pbb^0])
=2[\Pbb^1]+10 [\Pbb^0]\quad.
\]
On the other hand, denoting by $H$ the hyperplane class, $c(T\Pbb^3)=1+4H+6H^2+4H^3$;
and $X\cap Y$ consists of two lines meeting at a point, a curve of degree~$2$ and 
topological Euler characteristic~$3$, and hence $\csm(X\cap Y)=2[\Pbb^1]+3[\Pbb^0]$.
Thus the right-hand side of~\eqref{eq:splayed} equals
\[
c(TV)\cap \csm(X\cap Y) = 2[\Pbb^1] + 11[\Pbb^0]\quad,
\]
verifying that~\eqref{eq:splayed} does {\em not\/} hold in this case.
\qede\end{example}

\subsection{}
Several particular cases of the Scholium are proven in~\cite{MR3106736}. In this paper
we prove~\eqref{eq:splayed} under a hypothesis generalizing all those particular cases,
but possibly more restrictive than splayedness.
We say that $X$ and $Y$ are `strongly splayed' if $X=D'_1\cap \cdots \cap D'_r$, 
$Y=D''_1\cap \cdots \cap D''_s$ where $\{D'_1,\dots,D'_r\}$, $\{D''_1,\dots,D''_s\}$ are 
splayed sets of hypersurfaces. For example, two hypersurfaces are strongly splayed
if and only if they are splayed. We do not know if splayed subschemes of higher codimension
are necessarily strongly splayed, and this seems an interesting question.

\begin{thmI}
Let $X$, $Y$ be strongly splayed subschemes of a nonsingular variety~$V$.
Then
\begin{equation}\label{eq:splayedthmI}
\csm(X) \cdot \csm(Y) = c(TV) \cap \csm(X\cap Y)\quad.
\end{equation}
\end{thmI}

\begin{example}\label{intro:example}
Let $X$ be the union of a $\Pbb^4$ and a transversal $\Pbb^3$ in $V=\Pbb^5$; we choose
coordinates $(x_0:\cdots:x_5)$ so that $X$ has ideal $(x_0(x_1,x_2))$. Let $Y$ be the 
quadric cone with ideal $(x_3^2+x_4^2+x_5^2)$. Both $X$ and $Y$ are singular, and
$X$ is reducible and not pure-dimensional; $X$ and $Y$ are strongly splayed.

The Macaulay2 code from~\cite{MR1956868} may be used to compute the 
$\csm$ classes of $X$ and $Y$:
\begin{align*}
\csm(X) &= [\Pbb^4] + 6[\Pbb^3]+13 [\Pbb^2]+13 [\Pbb^1]+ 6[\Pbb^0] \\
\csm(Y) &= 2[\Pbb^4] + 8[\Pbb^3]+13 [\Pbb^2]+11 [\Pbb^1]+ 5[\Pbb^0] 
\end{align*}
According to Theorem~I, 
\[
c(T\Pbb^5)\cap \csm(X\cap Y) = \csm(X)\cdot \csm(Y)= 2[\Pbb^3]+20 [\Pbb^2]
+87 [\Pbb^1]+219[\Pbb^0]
\]
from which
\[
\csm(X\cap Y) =  2[\Pbb^3]+8 [\Pbb^2]+9 [\Pbb^1]+ 5[\Pbb^0]\quad.
\]
This can be verified by again using~\cite{MR1956868}.
\qede
\end{example}

In fact, the proper level of generality for the result is that of {\em constructible functions:\/}
a $\csm$ class in $A_*V$ is defined for every constructible function on $V$; if $X$ is
a subvariety of $V$, $\csm(X)=\csm(\one_X)$, where $\one_X$ is the indicator function
of $X$. Intersection of varieties corresponds naturally to the {\em product\/} of the 
corresponding constructible functions.

\begin{thmII}
Let $\varphi$, $\psi$ be constructible functions on a nonsingular variety~$V$, 
and assume that $\varphi$ and $\psi$ are strongly splayed. Then
\begin{equation}\label{eq:splayedthmII}
\csm(\varphi) \cdot \csm(\psi) = c(TV) \cap \csm(\varphi\cdot \psi)\quad.
\end{equation}
\end{thmII}

The precise definition of `strongly splayed' in the context of constructible functions
is given in Definition~\ref{def:splayedcf}; it generalizes naturally the notion for 
subvarieties. Note that
\eqref{eq:splayedthmII} amounts to the statement that the assignment
\[
\varphi \mapsto c(TV)^{-1}\cap \csm(\varphi)
\]
of a class in $A_*V$ from a constructible function $\varphi$
`preserves multiplication' for strongly splayed constructible functions. Again, this is 
clearly false without some kind of transversality condition on the constructible 
functions. It would be interesting to determine weaker conditions than `strong
splayedness' guaranteeing that this multiplicativity property holds.

Our proofs of Theorems~I and~II rely on intersection-theoretic considerations based 
on a formula for the Chern-Schwartz-MacPherson class of a hypersurface 
from~\cite{MR2001i:14009}, and on a general statement about Segre classes proven 
in this note (Theorem~\ref{thm:Segreclassrel}, which should be of independent interest). 
A proof of the Scholium for {\em splayed\/} (rather than {\em strongly splayed\/})
subvarieties should result as a particular case of J\"org Sch\"urmann's
Verdier-Riemann-Roch theorem for Chern-Schwartz-MacPherson classes (\cite{Schup}).

\subsection{}
It is natural to ask whether a version of the Scholium holds for other characteristic 
classes for singular varieties. Substantial work has been carried out comparing the
Chern-Schwartz-MacPherson class to the {\em Chern-Fulton\/} class, another class
agreeing with the Chern class of the tangent bundle for nonsingular varieties.
See Example~4.2.6 (a) in \cite{85k:14004} for the definition (reproduced here
in~\S\ref{sec:cfu}). We denote this class by
$\cFu$. The difference $\csm(X)-\cFu(X)$ is called the {\em Milnor class\/} of $X$, 
since it generalizes Milnor numbers of isolated hypersurface singularities to arbitrary
singularities.

\begin{thmIII}
Let $X$, $Y$ be splayed subschemes of a nonsingular variety~$V$. Then
\begin{equation}\label{eq:splayedthmIII}
\cFu(X) \cdot \cFu(Y) = c(TV) \cap \cFu(X\cap Y)\quad.
\end{equation}
\end{thmIII}

The proof of this result also follows from Segre class considerations, in fact of a simpler
nature than those leading to Theorem~\ref{thm:Segreclassrel}.

It is also natural to ask whether~\eqref{eq:splayed} and~\eqref{eq:splayedthmIII} hold
when $X$ and $Y$ are `in general position'. The following is a prototype situation
where this can be established.

\begin{thmIV}
Let $V$ be a nonsingular variety, and let $X\subseteq V$ be a general very ample
divisor on $V$. Then for all subschemes $Y\subseteq V$,
\begin{align*}\label{eq:splayedthmIV}
c(X)\cdot \csm(Y) &= c(TV)\cap \csm(X\cap Y) \\
c(X)\cdot \cFu(Y) &= c(TV)\cap \cFu(X\cap Y)
\end{align*}
\end{thmIV}

Theorem~IV hints that a condition analogous to splayedness may satisfy results
along the lines of the Bertini or Kleiman-Bertini theorems. It would be interesting
to establish a precise result of this type.
In general, however,  a ``splayed'' Bertini theorem cannot hold, as the following 
example illustrates:

\begin{example}\label{ex:noBertini}
Let $X$ be the so-called \emph{4-lines divisor} in $\Cbb^3$, given by the polynomial 
$xy(x+y)(x+yz)$. It is well-known that $X$ is not analytically trivial along the $z$-axis, 
i.e., there is no analytic isomorphism between two hyperplane sections $X_{t_1}$ and 
$X_{t_2}$, where $X_t:=X \cap (\mathbb{C}^2 \times \{t\})$. If $X$ were splayed with 
a general hyperplane at a general point of the $z$-axis, then it would be possible to 
write the equation of the divisor only using two coordinates at that point. This would 
imply that nearby sections are analytically isomorphic.
\qede\end{example}

{\em Acknowledgements.} P.A.'s research is partially supported by a Simons collaboration 
grant. E.F. was partially supported by Austrian Science Fund (FWF) project J3326 and 
this material is based upon work supported by the National Science Foundation under 
Grant No.~0932078 000, while E.F.~was in residence at the Mathematical Science 
Research Institute (MSRI) in Berkeley, California, during the spring semester of 2013.

 Both authors are very grateful to J\"org Sch\"urmann for communications regarding
the material covered in this note.


\section{Proofs of Theorems~I and~II}

\subsection{Splayed blow-ups}\label{sec:splayed}
Throughout the paper, $V$ will denote a smooth complex algebraic variety;
several results extend without change to the context of nonsingular algebraic
varieties over an algebraically closed field of characteristic~$0$. (See e.g.,
\cite{MR1063344} for a treatment of Chern-Schwartz-MacPherson classes in 
this generality.)
We call two subschemes $Z_1, Z_2
\subseteq V$ \emph{splayed} if at every point $p$ in the intersection of $Z_1$
and $Z_2$ there is a local analytic isomorphism $\varphi: V\to V'\times V''$, 
and subschemes $Z_1'\subseteq V'$, $Z_2''\subseteq V''$ such that 
$Z_1=\varphi^{-1} (Z'_1\times V'')$ and $Z_2=\varphi^{-1}(V'\times Z''_2)$.
Equivalently, there are analytic coordinates for $V$ at $p$ such that $Z_1$
and $Z_2$ are defined in different sets of variables.

More generally, we will say that two sets of subschemes are splayed in $V$ 
if at each point there is a local analytic
isomorphism $\varphi$ as above, such that the schemes in the first set are
inverse images from the first factor of the product $V'\times V''$, and the schemes
in the second sets are inverse images from the second factor.

Denote by $\pi_i: \TV_i\to V$ the blowup of $V$ along $Z_i$. Denote further
by $\TV_{12}$ the blowup of $\TV_1$ along the inverse image $\pi^{-1}_1Z_2$
of $Z_2$ and by $\TV_{21}$ the blowup of $\TV_2$ along the inverse image
$\pi^{-1}_2Z_1$ of $Z_1$.
We begin by recalling the following fact, for which the splayedness assumption
on $Z_1$, $Z_2$ or the smoothness of $V$ are not needed.

\begin{prop}
The blow-ups $\TV_{12}$ and $\TV_{21}$ are isomorphic, and they are isomorphic
to the blow-up of $V$ along $Z_1\cup Z_2$, where the defining ideal sheaf of
$Z_1 \cup Z_2$ in $V$ is the product of the ideal sheaves defining $Z_1$ and
$Z_2$.
\end{prop}

\begin{proof}
Both statements follow from the universal property of blow-ups; see for 
example \cite{MR1730819}, Lemma IV-41 and \cite{MR2595553}, Lemma 3.2.
\end{proof}

We want to compare $\TV_{12}$ to the fiber product of $\TV_1$ and $\TV_2$.

By the universal property of fiber products, there is a unique morphism $\alpha:
\TV_{12} \to \TV_1\times_V \TV_2$:
\[
\xymatrix{ 
\TV_{12} \ar@/_/[ddr] \ar@/^/[drr] \ar[dr]|-{\exists ! \alpha} \\ 
& \TV_1 \times_V \TV_2 \ar[d] \ar[r] & \TV_2 \ar[d]^{\pi_2} \\ 
& \TV_1 \ar[r]_{\pi_1} & V 
}
\]

\begin{prop}\label{modfibp}
Let $V$ be an irreducible variety. Then $\alpha$ induces an isomorphism 
from $\TV_{12}$ to the unique irreducible component 
$\TV_1\hat\times_V \TV_2$ of $\TV_1\times_V \TV_2$ mapping dominantly to $V$.
\end{prop}

\begin{remark}
M.~Kwieci\'nski (\cite{kwiecinski1994transforme}) calls this irreducible component
the `modified fiber product', and observes that it is a product in the category of 
proper birational morphisms from varieties to $V$.
\qede\end{remark}

\begin{proof}
Let $\hat\alpha$ be the induced morphism $\TV_{12}\to \TV_1\hat\times_V\TV_2$.
Since $\pi_1^{-1}(Z_1)$ and $\pi_2^{-1}(Z_2)$ are Cartier divisors,
it follows that their inverse images in $\TV_1\hat\times_V \TV_2$ are Cartier 
divisors, and hence so is the inverse image of $Z_1\cup Z_2$.
By the universal property of blow-ups, we obtain a morphism
$\TV_1\hat\times_V \TV_2 \to \TV_{12}$, which is immediately checked to be
the inverse of $\hat\alpha$.
\end{proof}

\begin{corol}\label{splayedbu}
Assume that $V$ is nonsingular and $Z_1$ and $Z_2$ are splayed in $V$.
Then $\alpha$ is an isomorphism
$\TV_{12}\to \TV_1\times_V \TV_2$.
\end{corol}

\begin{proof}
We claim that $\TV_1\times_V \TV_2$ is irreducible. Indeed, it suffices 
to verify this fact locally analytically over every $p$ in $V$, so by
the splayedness condition we may assume $Z_1=Z_1'\times V''$,
$Z_2=V'\times Z'_2$ in $V=V'\times V''$. In this situation
$\TV_1\times_V \TV_2\cong B\ell_{Z'_1}V'\times B\ell_{Z'_2}''V''$.

Since $B\ell_{Z'_1}V'$ and $B\ell_{Z'_2}V''$ are irreducible, this
product is irreducible.

With notation as in Proposition~\ref{modfibp}, this shows that 
$\TV_1\times_V \TV_2=\TV_1\hat\times_V \TV_2$, and the result follows
then from the proposition.
\end{proof}

\begin{remark} \label{rem:blowupirreducible}
It is worth pointing out that $\alpha$ is not an isomorphism in general.
For example, let $V=\Abb^2$ and let let $Z_1=Z_2$ be the origin $p=(0,0)$.
Then $\TV_1\times_V \TV_2$ consists of two components: an isomorphic
copy of the blow-up of $V$ at $p$, and a component isomorphic to $E\times E$,
where $E$ is the exceptional divisor in $\TV_1=\TV_2$. (This is easily verified
by a computation with charts.) 

Tracing the proof of Proposition~\ref{modfibp}, the problem is that while the
inverse image of e.g., $Z_1$ in $\TV_1\times_V \TV_2$ is locally principal, it contains
a whole component of the fiber product (i.e., local generators of its ideal are zero-divisors), 
so this subscheme is not a Cartier divisor of the fiber product.
It is however a Cartier divisor in the {\em modified\/} fiber product.

On the other hand, $\alpha$ may be an isomorphism even if $Z_1$ and $Z_2$ are not splayed. For instance,
if $Z_1$ and $Z_2$ are any Cartier divisors, then all blow-ups are isomorphisms, and
so is the fiber product. For a more substantive example, 
take two coordinate axes $Z_1$, $Z_2$ in $V=\Abb^3$. It can easily be 
seen via computation in charts that $\TV_1\times_V \TV_2$ is irreducible and isomorphic to 
the blow-up of $V$ along $Z_1 \cup Z_2$. Thus in this case $\alpha$ is an isomorphism, 
although $Z_1$ and $Z_2$ are not splayed according to our definition. 
\qede\end{remark}

\begin{corol}\label{fiberint}
Let $Z_1$, $Z_2$ be splayed in $V$, and consider the blow-ups along $Z_1$, $Z_2$,
and $Z_1\cup Z_2$ as above.
\[
\xymatrix{
\TV_{12} \ar[d]_{\tilde \pi_2} \ar[r]^{\tilde \pi_1} & \TV_2 \ar[d]^{\pi_2} \\
\TV_1 \ar[r]_{\pi_1} & V
}
\]
Then the homomorphisms $\tilde \pi_{2*} {\tilde \pi_1}^*$ and $\pi_1^* \pi_{2*}$ from 
$A_*\TV_2$ to $A_*\TV_1$ coincide.
\end{corol}

\begin{proof}
The maps are all proper l.c.i.~morphisms, and the diagram is a fiber square
by Corollary~\ref{splayedbu}. By \cite{85k:14004}, Example~17.4.1 (a),
\[
\pi_1^* \pi_{2*}(\alpha) = \tilde \pi_{2*} (c_e(\cE)\cap {\tilde \pi_1}^*(\alpha))
\]
for all $\alpha\in A_*(\TV_2)$, where $\cE$ is an excess bundle and $e$ is the
difference in the codimensions of $\pi_1$ and $\tilde \pi_1$. Here both $\pi_1$ 
and $\tilde \pi_1$ are birational, so $e=0$, and $c_e(\cE)=1$, hence the equality
follows.
\end{proof}

\subsection{A Segre class formula}\label{sec:Segreclassformula}
The $\csm$ class of a hypersurface $D$ in a nonsingular variety may be
expressed in terms of the {\em Segre class\/} of the singularity
subscheme~$JD$ in $V$.  The precise relationship (from~\cite{MR2001i:14009}) 
will be recalled below. The hypersurface case of Theorem~I will then follow from 
a statement on Segre classes of singularity subschemes of splayed hypersurfaces. 
In this subsection we prove a more general form of this statement 
(Theorem~\ref{thm:Segreclassrel}).

{\em Reminder.\/} Segre classes are one of the ingredients of
Fulton-MacPherson intersection theory, and the reader is addressed to
Chapter~4 of~\cite{85k:14004} for a thorough treatment of these
classes. The following summary should suffice for the purpose of
this paper. The Segre class $s(S,X)$ of a proper subscheme $S$ of a 
scheme $X$ is the class in the Chow group of $S$ determined by the 
following properties:
\begin{itemize}
\item Birational invariance: If $f: X' \to X$ is a proper birational morphism,
then $s(S,X)=f_* s(f^{-1}(S), X')$ (Proposition 4.2(a) in~\cite{85k:14004});
\item If $S$ is a Cartier divisor in $X$, then $s(S,X)=[S]-[S]^2+[S]^3-\cdots$ 
(Corollary 4.2.2 in~\cite{85k:14004}).
\end{itemize}
We use the shorthand $\dfrac {[S]}{1+S}$ for the class $[S]-[S]^2+[S]^3-\cdots$.

By the first property, blowing-up $X$ along $S$ reduces the computation of
$s(S,X)$ to the computation of the Segre class for the exceptional divisor
in the blow-up, which may be performed by using the second property. 
In practice it is often very difficult to carry out this process, but useful formulas 
for Segre classes may be proven by using this strategy. The second property
is a particular case of the following fact:
\begin{itemize}
\item If $S$ is regularly embedded in $X$, with normal bundle $N_SX$, 
then $s(S,X) = c(N_SX)^{-1}\cap [S]$ (Corollary 4.2.1 in~\cite{85k:14004}).
\end{itemize}
Below, this will be used in order to compute the Segre class of the complete 
intersection of two hypersurfaces.

Let $Z_1,Z_2,V,$ etc.~be as in \S\ref{sec:splayed}. By the birational invariance
of Segre classes recalled above,
\[
\pi_{1*} s(\pi_1^{-1}(Z_2),\TV_1)=s(Z_2,V)\quad.
\]
In the splayed situation, a stronger statement holds.

\begin{lemma}\label{pullback}
Let $Z_1,Z_2$ be splayed in $V$ (as in \S\ref{sec:splayed}). Then
\[
s(\pi_1^{-1}(Z_2),\TV_1) = \pi_1^* s(Z_2,V)\quad.
\]
\end{lemma}

\begin{proof}
Consider the diagram
\[
\xymatrix{
\TV_{12} \ar[d]_{\tilde \pi_2} \ar[r]^{\tilde \pi_1} & \TV_2 \ar[d]^{\pi_2} \\
\TV_1 \ar[r]_{\pi_1} & V
}
\]
as in \S\ref{sec:splayed}. Let $E_2=\pi_2^{-1}(Z_2)$ be the exceptional divisor in $\TV_2$,
and let $E'_2=\tilde{\pi_1}^{-1}(E_2)=\tilde{\pi_2}^{-1}(\pi_1^{-1}(Z_2))$. By the birational
invariance of Segre classes,
\[
s(\pi_1^{-1}(Z_2),\TV_1)=\tilde\pi_{2*}\left(\frac{[E'_2]}{1+E'_2}\right)
=\tilde\pi_{2*}\tilde \pi_1^* \left(\frac{[E_2]}{1+E_2}\right)\quad.
\]
Since $Z_1$ and $Z_2$ are splayed, by Corollary~\ref{fiberint} this equals
\[
\pi_1^*\pi_{2*} \left(\frac{[E_2]}{1+E_2}\right) = \pi_1^* s(Z_2,V)
\]
as claimed.
\end{proof}

\begin{remark}\label{rem:beyondsplayed}
The equality stated in Lemma~\ref{pullback} does not hold in general: $V=\Abb^2$,
$Z_1=Z_2=$ the origin give a simple counterexample. It does hold whenever the
fiber product $\TV_1\times_V \TV_2$ is irreducible, as the arguments given above
show, and this may occur even if $Z_1$ and $Z_2$ are not splayed. For example,
if $Z_1$ is a hypersurface of $V$, then this condition is trivially satisfied regardless
of splayedness. For a more interesting example, two lines $Z_1$, $Z_2$ meeting 
at a point in $V=\Pbb^3$ are not splayed according to our definition, yet
$\TV_1\times_V \TV_2$ is irreducible (cf.~Remark \ref{rem:blowupirreducible}).
\qede\end{remark}

We will use Lemma~\ref{pullback} in the proof of the following more general
Segre class formula, which is the key technical result needed for the first proof of
Theorem~I.

Let $D_1$, $D_2$ be hypersurfaces in $V$, and let $Z_1\subseteq D_1$,
$Z_2\subseteq D_2$ be subschemes. At the level of ideal sheaves, we have
\[
\cI_{D_1,V}\subseteq \cI_1
\quad,\quad
\cI_{D_2,V}\subseteq \cI_2
\]
where $\cI_1=\cI_{Z_1,V}$, $\cI_2=\cI_{Z_2,V}$. We consider the subscheme $W$
of $V$ defined by the ideal sheaf
\[
\cI_{W,V}:=\cI_{D_1,V}\cdot \cI_2 + \cI_{D_2,V}\cdot \cI_1\quad.
\]
This subscheme is supported on $(D_1\cup Z_2)\cap (D_2\cup Z_1)=(D_1\cap D_2)
\cup (Z_1\cup Z_2)$, with a scheme structure depending subtly on $Z_1$ and $Z_2$. 
Under a splayedness assumption, we will obtain a relation between the Segre classes 
of $Z_1$, $Z_2$, and $W$. The relation is best expressed in terms of the following
notation: for $\iota: Z\subset V$ an embedding of schemes, let
\[
\shat (Z,V)=[V]-\iota_* s(Z,V)^\vee\quad.
\]
Here, the dual $(\cdot)^\vee$ changes the sign of components of odd codimension in $V$.

\begin{theorem}\label{thm:Segreclassrel}
Let $D_1$, $D_2$ be hypersurfaces of a smooth variety $V$, and $Z_1\subseteq D_1$,
$Z_2\subseteq D_2$, $W$ as above. Assume that $\{Z_1,D_1\}$ and $\{Z_2,D_2\}$
are splayed. Then
\begin{equation}\label{eq:Segreclassrel}
\frac{\shat(W,V)\otimes \cO(D_1+D_2)}{1+D_1+D_2}
=\left(\frac{\shat(Z_1,V)\otimes \cO(D_1)}{1+D_1}\right)\cdot
\left(\frac{\shat(Z_2,V)\otimes \cO(D_2)}{1+D_2}\right)
\end{equation}
in the Chow group of $V$.
\end{theorem}

\begin{remark}
In this statement we use the notation introduced in~\S2 of~\cite{MR96d:14004}: 
if $\cL$ is a line bundle on $V$ and $A=\sum a^{(i)}$ is a class in the Chow group, 
where $a^{(i)}$ has {\em co}dimension~$i$ in $V$, then
\[
A\otimes \cL := \sum \frac{a^{(i)}}{c(\cL)^i}\quad.
\]
The formula given in Theorem~\ref{thm:Segreclassrel} is a good example of the 
usefulness of this notation: the formula (and its proof) would look unintelligibly 
complicated if it were written out without adopting this shorthand.
The notation satisfies simple properties, see Propositions~1 and~2 
in~\cite{MR96d:14004}; these will be used liberally in what
follows. It is also useful to observe that if $A$ and $B$ are classes in
$A_*V$, then 
\[
(A\cdot B)\otimes \cL = (A\otimes \cL)\cdot (B\otimes \cL)
\]
(this is evident from the definition).
\qede\end{remark}

\begin{proof}[Proof of Theorem~\ref{thm:Segreclassrel}]
We consider a sequence of three blow-ups over $V$:
\[
\xymatrix{
\TV \ar[r]^-{\tilde \pi} & \TV_{12} \ar[r]^-{\tilde \pi_2} & \TV_1 \ar[r]^-{\pi_1} & V\quad:
}
\]
the blow-up $\pi_1$ of $V$ along $Z_1$, with exceptional divisor $E_1$; 
the blow-up $\tilde \pi_2$ of $\TV_1$ along $\pi_1^{-1}(Z_2)$, with exceptional divisor
$E'_2$; and the blow-up $\tilde \pi$ of $\TV_{12}$ along the intersection of the residual
subschemes of $\tilde\pi_2^{-1}(E_1)$, resp., $E'_2$ in the inverse images of $D_1$,
resp.,~$D_2$. Note that under our splayedness hypothesis this last center is a 
complete intersection of codimension~$2$.
We let $\Til E$ be the exceptional divisor of $\tilde \pi$.
For notational convenience, we often use the same notation for an object
and for its inverse image to a variety in the sequence: for instance, $E_1$
will also denote its inverse image $\tilde \pi^{-1} \tilde \pi_2^{-1} E_1$ in 
$\TV$.
Finally, $\pi$ will denote the composition $\pi_1\circ \tilde\pi_2\circ \tilde \pi:
\TV\to \TV_{12}\to \TV_1\to V$.

\begin{claim}\label{cla:idealeq}
$\pi^{-1} (W) = E_1 \cup E'_2 \cup \Til E$.
\end{claim}

The statement of this claim is that the ideal of $W$ pulls back to the product of
the ideals of (the inverse images of) $E_1$, $E'_2$, and $\Til E$ in $\TV$. 
In $\TV_1$,
\begin{align*}
\cI_{W,V}\cdot \cO_{\TV_1} &= \cI_{\Til D_1,\TV_1} \cI_{E_1,\TV_1} \cI_{\pi_1^{-1}(Z_2),
\TV_1} +\cI_{\pi_1^{-1}(D_2),\TV_1} \cdot \cI_{E_1,\TV_1} \\
&=\cI_{E_1,\TV_1}\cdot
(\cI_{\Til D_1,\TV_1} \cI_{\pi_1^{-1}(Z_2),\TV_1} +\cI_{\pi_1^{-1}(D_2),\TV_1})\quad,
\end{align*}
where $\Til D_1$ is the residual of $E_1$ in $\pi_1^{-1}(D_1)$. In $\TV_{12}$,
\[
\cI_{W,V}\cdot \cO_{\TV_{12}} =\cI_{\tilde \pi_2^{-1}(E_1),\TV_{12}}\cdot \cI_{E'_2,\TV_{12}} 
\cdot (\cI_{\tilde\pi_2^{-1}(\Til D_1),\TV_{12}} +\cI_{ \Til D_2,\TV_{12}})\quad,
\]
where $\Til D_2$ is the residual of $E'_2$ in the inverse image of $D_2$.
The ideal $\cI_{\tilde\pi_2^{-1}(\Til D_1),\TV_{12}} +\cI_{\Til D_2,\TV_{12}}$
defines the center of the third blow-up, so this shows that
\[
\cI_{W,V}\cdot \cO_{\TV} 
=\cI_{\tilde \pi^{-1} \tilde \pi_2^{-1}(E_1),\TV}\cdot \cI_{\tilde\pi^{-1}(E'_2),\TV} 
\cdot \cI_{\Til E,\TV}
\] 
as claimed. 

By the birational invariance of Segre classes,
\[
s(W,V) = \pi_* \frac{[E_1]+[E'_2]+[\Til E]}{1+ E_1 + E'_2 + \Til E}\quad,
\]
and therefore
\[
[V]-s(W,V) = \pi_* \left(\frac{1}{1+ E_1 + E'_2 + \Til E}\cap [\TV]\right)\quad.
\]
Using the $\otimes$ notation recalled after the statement of the proposition,
\begin{align*}
\frac 1{1+E_1+E'_2+\Til E}\cap [\TV]
&=\left(\frac{1-E_1-E'_2}{1+\Til E}\cap [\TV]\right) \otimes \cO(E_1+E'_2) \\
&=\left((1-E_1-E'_2)\left(1-\frac{\Til E}{1+\Til E}\right)\cap [\TV]\right)\otimes 
\cO(E_1+E'_2) \\
&=\frac 1{1+E_1+E'_2}\cap \left([\TV] - \frac{[\Til E]}{1+\Til E}
\otimes \cO(E_1+E'_2)\right)\quad.
\end{align*}
The term $[\Til E]/(1+\Til E)$ pushes forward to the Segre class of the center
of the third blow-up, which is the intersection of the (inverse images of the) 
residual of $E_1$ in $D_1$, with class $D_1-E_1$, and of the residual of 
$E'_2$ in $D_2$, with class $D_2-E'_2$. The intersection is regularly embedded 
in $\TV_{12}$, as noted earlier, and its Segre class equals the inverse Chern 
class of its normal bundle (by the third property of Segre classes recalled above):
\[
\tilde \pi_* \left(\frac{\Til E}{1+\Til E}\right) = \frac{[D_1-E_1]\cdot [D_2-E'_2]}
{(1+D_1-E_1)(1+D_2-E'_2)}\quad,
\]
where evident pull-backs are omitted for notational simplicity.
Using this fact, properties of the $\otimes$ notation, and the projection formula,
\begin{multline*}
\tilde\pi_* \left(
\frac 1{1+E_1+E'_2+\Til E}\cap [\TV]
\right) \\
=\frac 1{1+E_1+E'_2}\cap \left([\TV_{12}] - \frac{[D_1-E_1]\cdot [D_2-E'_2]}
{(1+D_1-E_1)(1+D_2-E'_2)}\otimes \cO(E_1+E'_2)\right) \\
=\frac 1{1+E_1+E'_2}\cap \left([\TV_{12}] - \frac{[D_1-E_1]\cdot [D_2-E'_2]}
{(1+D_1+E'_2)(1+D_2+E_1)}\right)
\end{multline*}

A remarkable cancellation (and again the projection formula) now gives
\begin{multline*}
\tilde\pi_* \left(\frac{1}{(1+D_1+D_2)(1+ E_1 + E'_2 + \Til E)}\cap [\TV]\right) \\
=\frac 1{(1+D_1+D_2)(1+E_1+E'_2)}\cap \left(1 - \frac{(D_1-E_1)\cdot (D_2-E'_2)}
{(1+D_1+E'_2)(1+D_2+E_1)}\right)\cap [\TV_{12}] \\
=\frac{[\TV_{12}]}{(1+D_1+E'_2)(1+D_2+E_1)}\quad.
\end{multline*}
Summarizing, we have shown that
\begin{equation}\label{eq:intermed}
\frac{[V]-s(W,V)}{1+D_1+D_2} = \pi_{1*}\tilde\pi_{2*} 
\left(\frac{[\TV_{12}]}{(1+D_1+E'_2)(1+D_2+E_1)}\right)\quad.
\end{equation}
In order to evaluate the right-hand side, note that
\[
\frac{[\TV_{12}]}{1+D_1+E'_2} 
=\frac 1{1+D_1}\left(\frac {[\TV_{12}]}{1+E'_2}\otimes \cO(D_1)\right)
= \frac 1{1+D_1}\left([\TV_{12}] - \frac {[E'_2]}{1+E'_2}\otimes \cO(D_1)\right)
\]
Pushing this forward by $\tilde\pi_2$ shows that
\[
\tilde\pi_{2*}\left( \frac{[\TV_{12}]}{1+D_1+E'_2}\right)
= \frac {[\TV_1] - s(\pi_1^{-1} Z_2,\TV_1)\otimes \cO(D_1)}{1+D_1}\quad.
\]
Since $Z_1$ and $Z_2$ are splayed, by Lemma~\ref{pullback} this may be rewritten as
\[
\tilde\pi_{2*}\left( \frac{[\TV_{12}]}{1+D_1+E'_2}\right)
= \pi_1^*\left( \frac {[V] - s(Z_2,V)\otimes \cO(D_1)}{1+D_1}\right)\quad.
\]
By the projection formula and \eqref{eq:intermed} we have
\[
\frac{[V]-s(W,V)}{1+D_1+D_2} = \pi_{1*} \left(\frac{[\TV_1]}{1+D_2+E_1}\right)\cdot
\frac {[V] - s(Z_2,V)\otimes \cO(D_1)}{1+D_1}\quad.
\]
The last push-forward is handled similarly to the previous one, giving
\[
\pi_{1*} \left(\frac{[\TV_1]}{1+D_2+E_1}\right)
= \frac {[V] - s(Z_1,V)\otimes \cO(D_2)}{1+D_2}\quad.
\]
Therefore,
\[
\frac{[V]-s(W,V)}{1+D_1+D_2} = \frac {([V] - s(Z_1,V))\otimes \cO(D_2)}{1+D_2}
\cdot \frac {([V] - s(Z_2,V))\otimes \cO(D_1)}{1+D_1}\quad.
\]
The stated formula follows from this by taking duals and tensoring by 
$\cO(D_1+D_2)$.
\end{proof}

The argument shows that the formula in Theorem~\ref{thm:Segreclassrel} holds as
soon as $\TV_1\times_V \TV_2$ is irreducible (cf.~Remark~\ref{rem:beyondsplayed})
and the residuals of $E_1$ in $\pi_1^{-1}(D_1)$ and $E'_2$ in 
$\tilde\pi_2^{-1}\pi_1^{-1}(D_2)$ have no common components. While we focus on
splayedness in this paper, the formula in Theorem~\ref{thm:Segreclassrel}
has a substantially more general scope.

\begin{example}
Let $Z_1$, $Z_2$ be two lines in $V=\Pbb^3$ intersecting at a point. Then $Z_1$ and $Z_2$
are not splayed according to our definition, but $\TV_1\times_V \TV_2$ is irreducible
(Remark~\ref{rem:blowupirreducible}). Choosing coordinates $(x_0: \ldots:x_3)$, we may 
assume that $Z_1$ has the ideal $(x_0,x_1)$ and $Z_2$ has the ideal $(x_0,x_2)$. 
Then $Z_i$ is contained in $D_i=\{x_i=0\}$; a computation shows that the relevant residuals
have no common components.
A direct computation of Segre classes, which may for example be carried out 
using~\cite{MR1956868}, confirms that formula (\ref{eq:Segreclassrel}) does hold.
\qede\end{example}

\begin{example}
If $Z_1=Z_2=\emptyset$, then $W=D_1\cap D_2$.
Assume that $D_1$ and $D_2$ have no common components, so that $W$  
is a codimension~$2$ local complete intersection with normal bundle
$\cO(D_1)\oplus \cO(D_2)$. 
This is of course the case if $D_1$ and $D_2$ are splayed, and considerably more
generally. We have
\begin{align*}
\shat(W,V) &=[V]-\left(
\frac{D_1\cdot D_2}{(1+D_1)(1+D_2)}\cap [V]
\right)^\vee 
=\left(1-\frac{D_1\cdot D_2}{(1-D_1)(1-D_2)}\right)\cap [V] \\
&=\frac{ 1-D_1-D_2}{(1-D_1)(1-D_2)}\cap [V]
\end{align*}
and hence
\[
\shat(W,V) \otimes \cO(D_1+D_2) =\frac{1+D_1+D_2}{(1+D_1)(1+D_2)}\cap [V]
\]
(use Proposition~1 from \cite{MR96d:14004}).
Formula~\eqref{eq:Segreclassrel} follows immediately in this case. 

The reader is encouraged to consider the opposite extreme $Z_1=D_1$, $Z_2=D_2$,
and verify that \eqref{eq:Segreclassrel} reduces to $[V]=[V]\cdot [V]$ in this case (regardless
of splayedness).
\qede\end{example}

\subsection{Chern classes of hypersurface complements}
For a rapid review of {\em Chern-Schwartz-MacPherson\/} ($\csm$) classes, we 
address the reader to \S3.1 of \cite{MR3106736} and references therein. Briefly,
every locally closed subset $U$ of a complete variety $V$ determines a class
$\csm(U)$ in the Chow group of $V$, such that if $U=Z$ is a nonsingular closed
subvariety, then $\csm(Z)$ equals the push-forward to $V$ of the total Chern
class of the tangent bundle to $Z$. This notion is functorial in a strong sense,
and satisfies an inclusion-exclusion property: if $U_1$, $U_2$ are locally 
closed in $V$, then
\[
\csm(U_1\cup U_2) = \csm(U_1)+\csm(U_2) - \csm(U_1\cap U_2)\quad.
\]
The classes arose in seminal work of Marie-H\'el\`ene Schwartz (\cite{MR35:3707}, 
\cite{MR32:1727}) and Robert MacPherson (\cite{MR0361141}). See Example~19.1.7 
in~\cite{85k:14004} for an efficient statement of MacPherson's definition and result.

We will use the following formula computing the $\csm$ class of a hypersurface $D$ in
a nonsingular variety $V$ in terms of the Segre class of the {\em singularity 
subscheme\/}~$JD$, locally defined (as a subscheme of $D$) by the partial derivatives 
of a local equation for~$D$.

\begin{lemma}[\cite{MR2001i:14009}, Theorem~I.4]\label{lem:csmfromsegre}
Let $D$ be a hypersurface of a nonsingular variety $V$, with singularity subscheme $JD$. 
Then
\begin{equation}\label{eq:csmfromsegre}
\csm(D)= c(TV)\cap \left(s(D,V) + c(\cO(D))^{-1}\cap (s(JD,V)^\vee \otimes
\cO(D))\right)\quad.
\end{equation}
\end{lemma}
This statement again uses the operations $\cdot^\vee$, $\otimes$ employed in 
\S\ref{sec:Segreclassformula}. In terms of the notation introduced before the 
statement of Theorem~\ref{thm:Segreclassrel}, \eqref{eq:csmfromsegre}
is equivalent to
\begin{equation}\label{eq:csmfromshat}
\csm(V\smallsetminus D)=c(TV)\cap \frac{\shat(JD,V)\otimes \cO(D)}{1+D}\quad.
\end{equation}

Now suppose that $D_1$ and $D_2$ are splayed divisors in $V$, and let $D=D_1\cup D_2$.
Note that this implies that $\{D_1,JD_1\}$ and $\{D_2,JD_2\}$ are splayed.
It is clear set-theoretically that $JD$ is supported on $(D_1\cap D_2)\cup (JD_1 \cup JD_2)$.
The splayedness condition implies that the scheme structure of $JD$ on this union 
matches the one studied in \S\ref{sec:Segreclassformula} {\it vis-a-vis\/} $W$, $Z_1$, $Z_2$.

\begin{lemma}\label{lem:faber}
With notation as above, $D_1$ and $D_2$ are splayed if and only if
\[
\cI_{JD,V}=\cI_{D_1,V}\cdot \cI_{JD_2,V} + \cI_{D_2,V}\cdot \cI_{JD_1,V}\quad.
\]
\end{lemma}

This is a restatement of Corollary~2.6 in~\cite{MR3106736}. Only the `only if' part
will be needed here.

\begin{corol}\label{cor:shatcomplement}
Let $D_1$, $D_2$ be splayed hypersurfaces in a smooth variety $V$, and let
$D$ be $D_1\cup D_2$. Then
\[
\frac{\shat(JD,V)\otimes \cO(D)}{1+D}
=\left(\frac{\shat(JD_1,V)\otimes \cO(D_1)}{1+D_1}\right)\cdot
\left(\frac{\shat(JD_2,V)\otimes \cO(D_2)}{1+D_2}\right)
\]
in $A_*V$.
\end{corol}

\begin{proof}
This follows from Lemma~\ref{lem:faber} and Theorem~\ref{thm:Segreclassrel},
since if $D_1$ and $D_2$ are splayed, then so are $\{D_1,JD_1\}$ and $\{D_2,JD_2\}$.
\end{proof}

\begin{corol}\label{cor:csmcomplement}
Let $D_1$, $D_2$ be splayed hypersurfaces in a smooth variety $V$, and let
$D$ be $D_1\cup D_2$. Then
\begin{equation}\label{eq:complementrel}
c(TV)\cap \csm(V\smallsetminus D) = \csm(V\smallsetminus D_1)\cdot 
\csm(V\smallsetminus D_2)\quad.
\end{equation}
\end{corol}

\begin{proof}
This follows from Corollary~\ref{cor:shatcomplement} and \eqref{eq:csmfromshat}.
\end{proof}

Formula~\eqref{eq:complementrel} was proposed in \cite{MR3106736}, where it
was observed that under a strong freeness assumption on the divisors it follows from 
an analogous formula for Chern classes of sheaves of logarithmic differentials  
(Proposition~3.2 in \cite{MR3106736}). Several other particular instances of the formula 
are studied in~\S3 of~\cite{MR3106736}. Corollary~\ref{cor:csmcomplement} proves
the formula without extraneous assumptions.

For splayed divisors, Theorem~I follows immediately from Corollary~\ref{cor:csmcomplement}
and the inclusion-exclusion property of $\csm$ classes.

\begin{theorem}[Theorem~I, hypersurface case]\label{thm:hypcase}

Let $D_1$, $D_2$ be splayed divisors in a smooth variety~$V$. Then
\[
\csm(D_1)\cdot \csm(D_2) = c(TV)\cap \csm(D_1\cap D_2)
\]
in $A_*V$.
\end{theorem}

\begin{proof}
With $D=D_1\cup D_2$ as in Corollary~\ref{cor:csmcomplement}, and noting that
$c(TV)\cap \alpha = c(V)\cdot \alpha$ for all $\alpha\in A_*V$,
\begin{align*}
\csm(D_1)\cdot \csm(D_2) &= (c(V)-\csm(V\smallsetminus D_1))\cdot
(c(V)-\csm(V\smallsetminus D_2)) \\
&=c(TV)\cap \left(c(V)-\csm(V\smallsetminus D_1)-\csm(V\smallsetminus D_2)\right)\\
&\qquad +\csm(V\smallsetminus D_1)\cdot \csm(V\smallsetminus D_2) \\
&=c(TV)\cap \left(\csm(D_1)+\csm(D_2)-c(V)\right)+c(TV)\cap \csm(V\smallsetminus D) \\
&=c(TV)\cap \left(\csm(D_1)+\csm(D_2)-\csm(D)\right) \\
&=c(TV)\cap \csm(D_1\cap D_2)
\end{align*}
where the last equality follows by inclusion-exclusion.
\end{proof}

\subsection{
Strongly splayed varieties and constructible functions}\label{sec:endfirstproof}
We say that two subvarieties $Z_1$, $Z_2$ of $V$ are {\em strongly\/} splayed if
$Z_1=D'_1\cap \cdots \cap D'_r$, $Z_2=D''_1\cap \cdots \cap D''_s$ where $D'_i$, $D''_j$
are hypersurfaces, and $\{D'_1,\dots,D'_r\}$, $\{D''_1,\dots,D''_s\}$ are splayed in
the sense of \S\ref{sec:splayed}. We do not know if splayed subvarieties are necessarily
strongly splayed; the distinction if of course immaterial for hypersurfaces.

We can also consider this notion for {\em constructible functions.\/} By definition, every
constructible function can be written as a linear combination of indicator functions of
closed subvarieties. Since every subvariety is an intersection of hypersurfaces, it follows
that every constructible function may be written as an integer linear 
combination of indicator functions of {\em hypersurfaces.\/}

\begin{defin}\label{def:splayedcf}
Two constructible functions $\varphi$, $\psi$ are {\em strongly splayed\/} if they admit
representations
\begin{equation}\label{eq:defss}
\varphi = \sum_i a'_i D'_i\quad,\quad \psi =\sum_j a''_j D''_j
\end{equation}
with $\{D'_1,\dots,D'_r\}$, $\{D''_1,\dots,D''_s\}$ splayed sets of hypersurfaces
and $a_i', a_i''\in \Zbb$.
\qede
\end{defin}

Thus, if $Z_1$ and $Z_2$ are strongly splayed, then so are the corresponding
indicator functions $\one_{Z_1}, \one_{Z_2}$.

\begin{thmII}
Let $\varphi$, $\psi$ be strongly splayed constructible functions on a nonsingular 
variety~$V$. Then
\[
\csm(\varphi) \cdot \csm(\psi) = c(TV) \cap \csm(\varphi\cdot \psi)\quad.
\]
\end{thmII}

\begin{proof}
We will prove this statement by induction on the number of splayed hypersurfaces
needed to define $\varphi$, $\psi$ (as in~\eqref{eq:defss}).
More precisely, assume that the equality
\[
\csm(\varphi) \cdot \csm(\psi) = c(TV) \cap \csm(\varphi\cdot \psi)\quad.
\]
is known whenever $\varphi = \sum_{i=1}^r a'_i D'_i, \psi =\sum_{j=1}^s a''_j D''_j$
for a given pair $(r,s)$ of positive integers, with $\{D'_1,\dots,D'_r\}$ and $\{D''_1,\dots,D''_s\}$
splayed, and for all pairs preceding $(r,s)$ in the lexicographic order. 
We will show that the equality is then also true
for $(r+1,s)$. Since the statement is true for $(r,s)=(1,1)$ by Theorem~\ref{thm:hypcase}
(and $\cdot$ is symmetric), the general case follows by induction. Thus we are reduced to
showing that
\[
\csm(a\one_D+\varphi) \cdot \csm(\psi) = c(TV) \cap \csm((\one_D+\varphi)\cdot \psi)
\]
with $\varphi$ and $\psi$ as above, under the assumption that $\{D,D'_1,\dots,D'_r\}$ and 
$\{D''_1,\dots,D''_s\}$ are splayed. Since $\csm$ is linear,
\begin{align*}
\csm(a\one_D+\varphi) \cdot \csm(\psi) &= a\,\csm(\one_D)\cdot \csm(\psi)
+\csm(\varphi) \cdot \csm(\psi) \\
&= a\,c(TV)\cap \csm(\one_D\cdot \psi)+c(TV)\cap \csm(\varphi\cdot \psi) \\
\intertext{by the induction hypothesis}
&= c(TV)\cap (a\,\csm(\one_D\cdot \psi)+\csm(\varphi\cdot \psi)) \\
&= c(TV) \cap \csm((a\one_D+\varphi)\cdot \psi)
\end{align*}
as needed.
\end{proof}

Theorem~II implies the full statement of Theorem~I from the introduction. Indeed,
for $\varphi=\one_X$, $\psi=\one_Y$, under the assumption that $X$ and $Y$
(and hence $\varphi$, $\psi$) are strongly splayed, Theorem~II gives
\[
\csm(\one_X)\cdot \csm(\one_Y) = c(TV)\cap \csm(\one_X\cdot \one_Y)\quad,
\]
which gives~\eqref{eq:splayedthmI} as $\one_X\cdot \one_Y = \one_{X\cap Y}$.


\section{Proof of Theorem~III}\label{sec:cfu}

If $X$ is a subscheme of a nonsingular variety, the {\em Chern-Fulton\/} class of $X$, 
$\cFu(X)$, is defined by
\[
\cFu(X):= c(TV)\cap s(X,V)\quad.
\]
W.~Fulton introduced this class in~\cite{85k:14004}, Example~4.2.6 (a), and proved that
it is in fact independent of the choice of the ambient nonsingular variety $V$. If $X$ is
itself nonsingular, then $s(X,V)=c(N_XV)^{-1}\cap [X]$ (\S\ref{sec:Segreclassformula}),
so that $\cFu(X)=c(X)=\csm(X)$ in this case. The classes $\csm(X)$ and $\cFu(X)$ differ
in general; for example, $\cFu(X)$ is sensitive to the scheme structure of $X$, while
$\csm(X)$ only depends on the support of~$X$.

Theorem~III is a straightforward consequence of the following multiplicative formula
for Segre classes of splayed subschemes.

\begin{lemma}\label{lem:multsegre}
Let $Z_1$, $Z_2$ be splayed subschemes of a nonsingular variety $V$. Then
\begin{equation}\label{eq:multsegre}
s(Z_1\cap Z_2,V) = s(Z_1,V)\cdot s(Z_2,V)
\end{equation}
in $A_*(Z_1\cap Z_2)$.\end{lemma}

\begin{proof}
Consider again the fiber square of blow-ups
\[
\xymatrix{
\TV_{12} \ar[d]_{\tilde \pi_2} \ar[r]^{\tilde \pi_1} & \TV_2 \ar[d]^{\pi_2} \\
\TV_1 \ar[r]_{\pi_1} & V
}
\]
as in \S\ref{sec:splayed}. Let $E_1$ be the exceptional divisor in $\TV_1$, $E_2$
the divisor in $\TV_2$. By splayedness, the inverse images $\tilde \pi_2^{-1}(E_1)$
and $\tilde \pi_1^{-1}(E_2)$ have no components in common. 
Thus the inverse image
of $Z_1\cap Z_2$ in $\TV_{12}$ is the complete intersection of $\tilde \pi_2^{-1}(E_1)$
and $\tilde \pi_1^{-1}(E_2)$. 
Therefore
\begin{align*}
s(Z_1\cap Z_2,V) &= (\pi_1\circ\tilde\pi_2)_* \frac{\tilde \pi_2^*(E_1) \cdot \tilde \pi_1^*(E_2)}
{(1+\tilde \pi_2^*(E_1))(1+\tilde \pi_1^*(E_2))} \\
&=\pi_{1*}\left(\frac{E_1}{1+E_1}\cdot \tilde\pi_{2*} \tilde\pi_1^* \frac{E_2}{1+E_2}\right)\\
\intertext{by the projection formula}
&=\pi_{1*}\left(\frac{E_1}{1+E_1}\cdot \pi_1^* \pi_{2*} \frac{E_2}{1+E_2}\right)\\
\intertext{since the diagram is a fiber square}
&=\pi_{1*}\left(\frac{E_1}{1+E_1}\cdot \pi_1^* s(Z_2,V)\right)\\
&=s(Z_1,V)\cdot s(Z_2,V)
\end{align*}
again by the projection formula.
\end{proof}

\begin{remark}
Formula~\eqref{eq:multsegre} also follows formally by setting $D_1=D_2=0$ in 
\eqref{eq:Segreclassrel}; note that if $\cI_{D_1,V}$ and $\cI_{D_1,V}$ are trivial, then the 
the scheme $W$ appearing in Theorem~\ref{thm:Segreclassrel} equals $Z_1\cap Z_2$.
However, then the proof given for Theorem~\ref{thm:Segreclassrel} does then not work: 
one has to assume that $D_1$ 
and $D_2$ are hypersurfaces containing $Z_1$, $Z_2$ respectively, and this is in general
incompatible with assuming that their classes vanish.
Also, \eqref{eq:Segreclassrel} only holds in $A_*V$,
while~Lemma~\ref{lem:multsegre} proves~\eqref{eq:multsegre} in $A_*(Z_1\cap Z_2)$.
\qede\end{remark}

Theorem~III, stated in the introduction, follows immediately from Lemma~\ref{lem:multsegre}:
assuming $X$ and $Y$ are splayed, 
\begin{align*}
\cFu(X) \cdot \cFu(Y) &= (c(TV)\cap s(X,V))\cdot (c(TV)\cap s(Y,V)) \\ 
&= c(TV)\cap (c(TV)\cap (s(X,V)\cdot s(Y,V))) \\
&= c(TV)\cap (c(TV)\cap s(X\cap Y,V)) \\
&= c(TV)\cap \cFu(X\cap Y)\quad.
\end{align*}

\begin{example}
With $X$ and $Y$ as in Example~\ref{intro:example}, we have
\begin{align*}
\cFu(X) &= [\Pbb^4] + 6[\Pbb^3]+11 [\Pbb^2]+12 [\Pbb^1]+ 3[\Pbb^0] \\
\cFu(Y) &= 2[\Pbb^4] + 8[\Pbb^3]+14 [\Pbb^2]+12 [\Pbb^1]+ 6[\Pbb^0] 
\end{align*}
(obtained using the code from~\cite{MR1956868}). According to Theorem~III,
\[
c(T\Pbb^5)\cap\cFu(X\cap Y) = \cFu(X)\cdot \cFu(Y)= 2[\Pbb^3]+20 [\Pbb^2]
+84 [\Pbb^1]+208[\Pbb^0]
\]
from which
\[
\csm(X\cap Y) =  2[\Pbb^3]+8 [\Pbb^2]+6 [\Pbb^1]+ 12[\Pbb^0]\quad.
\]
Again, this can be verified in this example by using the code in~\cite{MR1956868}.
\qede\end{example}


\section{Proof of Theorem~IV}

We now assume that $X$ is a general section of a very ample line bundle on $V$;
in particular, $X$ is itself nonsingular. If a `Bertini theorem for splayedness' held,
then one would expect that for any $Y\subseteq V$, the formulas established
in Theorem~I and Theorem~III would hold. We prove these formulas independently
of such Bertini statements (and without invoking splayedness);
 as we pointed out in the introduction, a simple-minded `splayed Bertini'
statement in fact does {\em not\/} hold (Example~\ref{ex:noBertini}).

Our main tool is again a formula for Segre classes, which we reproduce here for
the convenience of the reader.

\begin{lemma}\label{lem:segrebert}
Let $Z\subseteq W$ be schemes, and let $D$ be a Cartier divisor in $W$, meeting 
properly the support of every component of the normal cone of $Z$ in $W$. Then
\begin{equation}\label{eq:segrebert}
s(D\cap Z,D)=D\cdot s(Z,W)\quad.
\end{equation}
\end{lemma}

\begin{proof}
Under the hypothesis of this statement, the blow-up of $D$ along $D\cap Z$
is the inverse image of $D$ in the blow-up of $W$ along $Z$. The statement
follows then from the projection formula.
\end{proof}

The formula for the Chern-Fulton class in Theorem~IV follows easily. Indeed,
if $X$ is general and very ample, then it can be chosen to intersect properly 
the components of the normal cone of $Y$; further, $X$ is nonsingular, so 
applying~\eqref{eq:segrebert} and the definition of Chern-Fulton class,
\[
\cFu(X\cap Y)=c(TX)\cap s(X\cap Y,X)=c(TX)\cap (X\cdot s(Y,V))=c(X)\cdot s(Y,V)\quad.
\]
Thus
\[
c(TV)\cap \cFu(X\cap Y)=c(X)\cdot c(TV)\cap s(Y,V) = c(X)\cdot \cFu(Y)\quad,
\]
as stated in Theorem~IV.

For the proof of the corresponding statement about $\csm$ classes, after applying
a Veronese embedding we may assume that $V\subseteq \Pbb^n$ and $X$ is a
general hyperplane section. In this situation,
\begin{equation}\label{lem:hypsec}
\csm(X\cap Y) = \frac X{1+X} \cdot \csm(Y)\quad.
\end{equation}
This follows from Proposition~2.6 in~\cite{MR3031565}. (The proof of this proposition
may be summarized as follows: by inclusion-exclusion
it can be reduced to the case in which $Y$ is a hypersurface; using 
Lemma~\ref{lem:csmfromsegre}, the formula amounts then to a relation for Segre
classes that ultimately depends again on Lemma~\ref{lem:segrebert}.)
From~\eqref{lem:hypsec},
\[
c(TV)\cap \csm(X\cap Y) = \left(c(TV)\cap \frac X{1+X}\right)\cdot \csm(Y)
=c(X)\cdot \csm(Y)\quad,
\]
completing the proof of Theorem~IV.



\end{document}